\documentclass[12pt,reqno]{amsart}
\usepackage{amscd,amsmath,amsthm,amssymb}
\usepackage{color, comment}
\usepackage{pstricks}
\usepackage{stmaryrd}
\usepackage{tikz-cd}
\newcommand{\arrowIn}{
\tikz \draw[-stealth] (-1pt,0) -- (1pt,0);
}
\newcommand{\arrowOut}{
\tikz \draw[-stealth] (1pt,0) -- (-1pt,0);
}
\usepackage{url}
\tikzset{empty/.style={black, fill=white}}

\pgfdeclarelayer{background}
\pgfdeclarelayer{foreground}
\pgfsetlayers{background,main,foreground}

\usepackage{latexsym}
\usepackage{amsfonts,amsmath,mathtools}
\usepackage{graphics}
\usepackage{float}
\usepackage{enumitem}

\newpsstyle{fatline}{linewidth=1.5pt}
\newpsstyle{fyp}{fillstyle=solid,fillcolor=verylight}
\definecolor{verylight}{gray}{0.97}
\definecolor{light}{gray}{0.9}
\definecolor{medium}{gray}{0.85}
\definecolor{dark}{gray}{0.6}

\def\NZQ{\mathbb}               

\def\FF{{\NZQ F}}

\def\G{{\mathcal G}}

\def\pd{\textup{pd}}

\def\opn#1#2{\def#1{\operatorname{#2}}} 
\opn\chara{char} \opn\length{\ell} \opn\pd{pd} \opn\rk{rk}
\opn\projdim{proj\,dim} \opn\injdim{inj\,dim} \opn\rank{rank}
\opn\depth{depth} \opn\grade{grade} \opn\height{height}
\opn\embdim{emb\,dim} \opn\codim{codim}

\opn\Tr{Tr} \opn\bigrank{big\,rank}
\opn\superheight{superheight}\opn\lcm{lcm}
\opn\trdeg{tr\,deg}
\opn\reg{reg} \opn\lreg{lreg} \opn\ini{in} \opn\lpd{lpd}
\opn\size{size} \opn\sdepth{sdepth}
\opn\link{link}\opn\fdepth{fdepth}\opn\lex{lex}
\opn\tr{tr}
\opn\type{type}
\opn\gap{gap}
\opn\diam{diam}
\opn\Mod{Mod}

\DeclareMathOperator{\im}{im}
\DeclareMathOperator{\wim}{wim}
\opn\div{div} \opn\Div{Div} \opn\cl{cl} \opn\Cl{Cl}
\opn\Spec{Spec} \opn\Supp{Supp} \opn\supp{supp} \opn\Sing{Sing}
\opn\Ass{Ass} \opn\Min{Min}\opn\Mon{Mon}
\opn\Ann{Ann} \opn\Rad{Rad} \opn\Soc{Soc}
\opn\Im{Im} \opn\Ker{Ker} \opn\Coker{Coker} \opn\Am{Am}
\opn\Hom{Hom} \opn\Tor{Tor} \opn\Ext{Ext} \opn\End{End}
\opn\Aut{Aut} \opn\id{id}

\opn\nat{nat}
\opn\pff{pf}
\opn\Pf{Pf} \opn\GL{GL} \opn\SL{SL} \opn\mod{mod} \opn\ord{ord}
\opn\Gin{Gin} \opn\Hilb{Hilb}\opn\sort{sort}
\opn\PF{PF}\opn\Ap{Ap}
\opn\dist{dist}
\opn\aff{aff}
\opn\relint{relint} \opn\st{st}
\opn\lk{lk} \opn\cn{cn} \opn\core{core} \opn\vol{vol}  \opn\inp{inp} \opn\nilpot{nilpot}
\opn\link{link} \opn\star{star}\opn\lex{lex}\opn\set{set}
\opn\width{wd}
\opn\Fr{F}
\opn\QF{QF}
\opn\G{G}
\opn\type{type}\opn\res{res}
\opn\conv{conv}
\opn\sr{sr}
\opn\gr{gr}

\def\pot#1#2{#1[\kern-0.28ex[#2]\kern-0.28ex]}

\opn\dirlim{\underrightarrow{\lim}}
\opn\inivlim{\underleftarrow{\lim}}
\let\to=\rightarrow

\def\Implies{\ifmmode\Longrightarrow \else
	\unskip${}\Longrightarrow{}$\ignorespaces\fi}
\def\implies{\ifmmode\Rightarrow \else
	\unskip${}\Rightarrow{}$\ignorespaces\fi}
\def\iff{\ifmmode\Longleftrightarrow \else
	\unskip${}\Longleftrightarrow{}$\ignorespaces\fi}

\let\:=\colon
\newtheorem{Theorem}{Theorem}[section]
\newtheorem{Lemma}[Theorem]{Lemma}
\newtheorem{Corollary}[Theorem]{Corollary}
\newtheorem{Proposition}[Theorem]{Proposition}
\newtheorem{Remark}[Theorem]{Remark}

\newtheorem{Example}[Theorem]{Example}

\newtheorem{Definition}[Theorem]{Definition}

\newtheorem{Question}[Theorem]{Question}

\let\epsilon\varepsilon
\let\kappa=\varkappa
\def\qed{\ifhmode\textqed\fi
	\ifmmode\ifinner\hfill\quad\qedsymbol\else\dispqed\fi\fi}
\def\textqed{\unskip\nobreak\penalty50
	\hskip2em\hbox{}\nobreak\hfill\qedsymbol
	\parfillskip=0pt \finalhyphendemerits=0}
\def\dispqed{\rlap{\qquad\qedsymbol}}

\begin{document}

	\title{Matching powers of monomial ideals and edge ideals of weighted oriented graphs}
	\author{Nursel Erey, Antonino Ficarra}
	
	\address{Nursel Erey, Gebze Technical University, Department of Mathematics, 41400 Gebze, Kocaeli, Turkey}
	\email{nurselerey@gtu.edu.tr}
	
	\address{Antonino Ficarra, Department of mathematics and computer sciences, physics and earth sciences, University of Messina, Viale Ferdinando Stagno d'Alcontres 31, 98166 Messina, Italy}
	\email{antficarra@unime.it}
	
	
	
	
	\maketitle
	
	\begin{abstract}
      We introduce the concept of matching powers of monomial ideals. Let $I$ be a monomial ideal of $S=K[x_1,\dots,x_n]$, with $K$ a field. The $k$th matching power of $I$ is the monomial ideal $I^{[k]}$ generated by the products $u_1\cdots u_k$ where $u_1,\dots,u_k$ is a sequence of support disjoint monomials contained in $I$. This concept naturally generalizes that of squarefree powers of squarefree monomial ideals. We study normalized depth function of matching powers of monomial ideals. We provide bounds for the regularity and projective dimension of edge ideals of weighted oriented graphs. When $I$ is a non-quadratic edge ideal of a weighted oriented graph which has no even cycles, we characterize when $I^{[k]}$ has a linear resolution. 
	\end{abstract}
	
	\section*{Introduction}
        Let $S=K[x_1,\dots,x_n]$ be the polynomial ring over a field $K$. Recall that the edge ideal of a finite simple graph $G$ with vertices $x_1,\dots ,x_n$ is generated by all the monomials $x_ix_j$ such that $\{x_i,x_j\}$ is an edge of $G$. The study of minimal free resolutions of edge ideals and their powers produced a great deal of interaction between combinatorics and commutative algebra. One of the most natural problems in this regard is to understand when those ideals, or more generally monomial ideals, have linear resolutions.  Although edge ideals with linear resolutions are combinatorially characterized by a famous result of Fröberg~\cite{F}, it is unknown in general when powers of edge ideals have linear resolutions. Herzog, Hibi and Zheng~\cite{HHZ2} showed that if an edge ideal has a linear resolution, then so does every power of it. It is their result that served as a starting point for the close examination of linear resolutions of powers of edge ideals by many researchers, resulting in several interesting results and conjectures. 
 
 For any squarefree monomial ideal $I$ of $S$, the $k$th squarefree power of $I$, denoted by $I^{[k]}$ is the monomial ideal generated by all squarefree monomials in $I^k$. Recently, squarefree powers of edge ideals were studied in \cite{BHZN18, CFL, EH2021, EHHM2022a, EHHM2022b, FHH23, SASF2022, SASF2023}. Determining linearity of minimal free resolutions of squarefree powers or finding their invariants is as challenging as those of ordinary powers although squarefree and ordinary powers have quite different behavior. In the case that $I$ is considered as edge ideal of a hypergraph $\mathcal{H}$, the minimal monomial generators of $I^{[k]}$ correspond to matchings of $\mathcal{H}$ of size $k$, which makes combinatorial aspect of squarefree powers interesting as well.

 This paper aims at presenting a wider framework for the study of squarefree powers by introducing a more general concept which we call matching powers. If $I$ is a monomial ideal of $S$, then the \emph{$k$th matching power} $I^{[k]}$ of $I$ is generated by the products $u_1\cdots u_k$ where $u_1,\dots,u_k$ is a sequence of monomials in $I$ with pairwise disjoint support. Indeed, if $I$ is a squarefree monomial ideal, then the $k$th squarefree power of $I$ is the same as the $k$th matching power of $I$. With this new concept, since we are no longer restricted to squarefree monomial ideals, we can consider not only edge ideals of simple graphs but also edge ideals of weighted oriented graphs. 
  
        We now discuss how the paper is organized. In Section~\ref{sec:1-EreyFic}, we summarize basic facts of the theory of matching powers. We define the normalized depth function $g_I$ of a monomial ideal $I$ in Definition~\ref{def:gI-MononomialCase}. This function generalizes the normalized depth function introduced in \cite{EHHM2022b} for squarefree monomial ideals. In Theorem~\ref{Thm:I^[nu(I)]Polymatroidal} we show that if $I$ is a quadratic monomial ideal, then the highest nonvanishing matching power of $I$ is polymatroidal. 

        In Section~\ref{sec:2-EreyFic}, we turn our attention to edge ideals of weighted oriented graphs. We make comparisons between homological invariants of matching powers $I(\mathcal{D})^{[k]}$ and $I(G)^{[k]}$, where $G$ is the underlying graph of a weighted oriented graph $\mathcal{D}$. We provide lower bounds for the regularity and projective dimension of $I(\mathcal{D})^{[k]}$ in Propositions \ref{prop:regularity lower bound} and \ref{prop:projective dimension lower bound}.
        
        In Section~\ref{sec:3-EreyFic}, we study linearly related matching powers. The main result of the section is Theorem~\ref{thm: linearly related only in matching power} which characterizes when $I(\mathcal{D})^{[k]}$ has a linear resolution or is linearly related provided that the underlying graph $G$ of $\mathcal{D}$ has no even cycles and $I(\mathcal{D})\neq I(G)$. In particular, this result combined with \cite[Theorem~41]{EH2021} gives a complete classification of weighted oriented forests $\mathcal{D}$ such that $I(\mathcal{D})^{[k]}$ has a linear resolution.

	\section{Matching Powers}\label{sec:1-EreyFic}
	
	Let $S=K[x_1,\dots,x_n]$ be the standard graded polynomial ring with coefficients in a field $K$. Let $I\subset S$ be a monomial ideal. We denote by $G(I)$ the minimal monomial generating set of $I$.  If $u$ is a monomial, we call $\supp(u)=\{i:x_i\ \textup{divides}\ u\}$ the \textit{support} of $u$. 
 The \textit{$k$th matching power} of $I$ is the monomial ideal defined as
 $$
	I^{[k]}\ =\ (f_1\cdots f_k\ :\ f_i\in G(I), \, \supp(f_i)\cap\supp(f_j)=\emptyset \text{ for all } 1\le i<j\le k ).
	$$
 
 Recall that $f_1,\dots,f_m$ is a \textit{regular sequence} (on $S$) if $f_i$ is a non zero--divisor on $S/(f_1,\dots,f_{i-1})$ for $i=1,\dots,m$. Therefore one can write 
$$
	I^{[k]}\ =\ (f_1\cdots f_k\ :\ f_i\in G(I), f_1,\dots,f_k\ \textit{is a regular sequence}).
	$$
We denote by $\nu(I)$ the \textit{monomial grade} of $I$, that is, the maximum size of a set of monomials in $I$ which are pairwise support disjoint. Note that $I^{[k]}\ne0$ if and only if $1\le k\le\nu(I)$.
	
	We define the \textit{support} of $I$ by $\supp(I)=\bigcup_{u\in G(I)}\supp(u)$. We say that $I$ is \textit{fully supported} if $\supp(I)=\{1,2,\dots,n\}$. From now on, we tacitly assume that all monomial ideals we consider are fully supported.
	
	\begin{Example}
		\rm \begin{enumerate}[label=(\roman*)]
			\item Let $I$ be a squarefree monomial ideal. Then, a product $u_1\cdots u_k$ with $u_i\in G(I)$ is in $I^{[k]}$ if and only if $u_1\cdots u_k$ is squarefree. Thus, in this case, $I^{[k]}$ is the usual \textit{$k$th squarefree power} of $I$ introduced in \cite{BHZN18}.
			\item Let $I$ be a complete intersection monomial ideal generated by $u_1,\dots,u_m$. Then $I^{[k]}=(u_{i_1}\cdots u_{i_k}:1\le i_1<\dots<i_k\le m)$ and $\nu(I)=m$.
			\item Let $(x_1^2,\,x_2^2,\,x_3^2,\,x_3x_4,\,x_5^5)$. Then $\nu(I)=4$ and
   \begin{align*}
       \phantom{aaaaa}I^{[2]}\ &=\ (x_1^2x_2^2,\,x_1^2x_3^2,\,x_1^2x_3x_4,\,x_1^2x_5^5,\,x_2^2x_3^2,\,x_2^2x_3x_4,\,x_2^2x_5^5,\,x_3^2x_5^5,\,x_3x_4x_5^5)\\
       I^{[3]}\ &=\ (x_1^2x_2^2x_3^2,\,x_1^2x_2^2x_3x_4,\,x_1^2x_2^2x_5^5,\,x_1^2x_3^2x_5^5\,,x_1^2x_3x_4x_5^5,\,x_2^2x_3^2x_5^5,\,x_2^2x_3x_4x_5^5),\\
       I^{[4]}\ &=\ (x_1^2x_2^2x_3^2x_5^5,\,x_1^2x_2^2x_3x_4x_5^5).
   \end{align*}
		\end{enumerate}
	\end{Example}\bigskip
 
\subsection*{Normalized depth function}

	For a monomial $u\in S$, $u\ne1$, the \textit{$x_i$-degree} of $u$ is defined as the integer
	$$
	\deg_{x_i}(u)\ =\ \max\{j\ge0:x_i^j\ \textup{divides}\ u\}.
	$$
	
	Let $I\subset S$ be a monomial ideal. The \textit{initial degree} of $I$, denoted by $\textup{indeg}(I)$ is the smallest degree of a monomial belonging to $I$. Following \cite{F2}, we define the \textit{bounding multidegree} of $I$ to be the vector
	$$
	{\bf deg}(I)\ =\ (\deg_{x_1}(I),\dots,\deg_{x_n}(I)),
	$$
	with
	$$
	\deg_{x_i}(I)\ =\ \max_{u\in G(I)}\deg_{x_i}(u),\ \ \textup{for all}\ \ \ 1\le i\le n.
	$$
 
	We provide a lower bound for the depth of $S/I^{[k]}$ in terms of the initial degree of $I^{[k]}$ and the bounding multidegree of $I$ as follows:
 
	\begin{Theorem}\label{Thm:ineqDepthIMon}
		Let $I\subset S$ be a monomial ideal. Then, for all $1\le k\le\nu(I)$, we have
		$$
		\depth(S/I^{[k]})\ \ge\ \textup{indeg}(I^{[k]})-1+(n-|{\bf deg}(I)|).
		$$
	\end{Theorem}
	\begin{proof}
		We divide the proof in three steps.\medskip\\
		{\bf (Step 1).} Let $J\subset S$ be a monomial ideal. We claim that
		$$
		\pd(J)\le|{\bf deg}(J)|-\textup{indeg}(J).
		$$
		To prove the assertion, we use the Taylor resolution. Let $\beta_{i,j}(J)$ be a non--zero graded Betti number with $i=\pd(J)$. Then $j\ge\textup{indeg}(J)+\pd(J)$. It follows from the Taylor resolution that the highest shift in the minimal resolution of $J$ is at most $|{\bf deg}(J)|$, see \cite[Theorem 1.3]{F2}. Thus, $|{\bf deg}(J)|\ge j$. Altogether, we obtain $|{\bf deg}(J)|\ge j\ge\textup{indeg}(J)+\pd(J)$ and the assertion follows.\medskip\\
		{\bf(Step 2).} We claim that $|{\bf deg}(I^{[k]})|\le|{\bf deg}(I)|$ for all $1\le k\le\nu(I)$. Indeed, we even show that $\deg_{x_\ell}(I^{[k]})\le\deg_{x_\ell}(I)$ for all $\ell$. A set of generators of $I^{[k]}$ is $$\Omega\ =\ \{u_1\cdots u_k\ :\ u_i\in G(I),\supp(u_i)\cap\supp(u_j)=\emptyset,1\le i<j\le k\}.$$ Thus, $G(I^{[k]})$ is a subset of $\Omega$. Hence, if $v\in G(I^{[k]})$, then $v=u_1\cdots u_k\in\Omega$. Let $x_\ell$ be a variable dividing $v$, then $x_\ell$ divides at most one monomial $u_i$, say $u_{i_{\ell}}$. Therefore, $\deg_{x_\ell}(v)\le\deg_{x_\ell}(u_{i_\ell})\le\deg_{x_\ell}(I)$ and the assertion follows.\medskip\\
		{\bf(Step 3).} By Steps 1 and 2 we have
		$$
			\pd(S/I^{[k]})\ \le\ |{\bf deg}(I^{[k]})|-\textup{indeg}(I^{[k]})+1\ \le\ |{\bf deg}(I)|-\textup{indeg}(I^{[k]})+1.
		$$
		The asserted inequality follows from the Auslander--Buchsbaum formula.
	\end{proof}
	
	As a consequence of Theorem \ref{Thm:ineqDepthIMon}, we can give the next definition:
	\begin{Definition}\label{def:gI-MononomialCase}
		\rm Let $I\subset S$ be a monomial ideal. For all $1\le k\le\nu(I)$, we set
		$$
		g_I(k)\ =\ \depth(S/I^{[k]})+|{\bf deg}(I)|-n-(\textup{indeg}(I^{[k]})-1),
		$$
		and call $g_I$ the \textit{normalized depth function} of $I$.
	\end{Definition}
	
	By Theorem \ref{Thm:ineqDepthIMon} the normalized depth function of $I$ is a nonnegative function. If $I\subset S$ is a squarefree monomial ideal, then ${\bf deg}(I)={\bf 1}=(1,\dots,1)$ and so
	$$
	g_{I}(k)=\depth(S/I^{[k]})-(\textup{indeg}(I^{[k]})-1)
	$$
	is the normalized depth function of $I$ introduced in \cite{EHHM2022b}. In the same paper, it was conjectured that $g_{I}$ is a nonincreasing function for any squarefree monomial ideal $I$. Later, Seyed Fakhari \cite{SASF2023-2} disproved the conjecture by constructing a family of cubic squarefree monomial ideals where $g_I(2)-g_I(1)$ can be arbitrarily large. On the other hand, it is still unknown whether the conjecture holds for edge ideals.

 Next, we will see how the normalized depth function and matching powers of a monomial ideal are related to those of its polarization. Before we proceed to this, recall that the \textit{polarization} of a monomial $u=x_1^{b_1}\cdots x_n^{b_n}\in S$ is the monomial
    $$
    u^\wp=\prod_{i=1}^n(\prod_{j=1}^{b_i}x_{i,j})=\prod_{\substack{1\le i\le n\\ b_i>0}}x_{i,1}x_{i,2}\cdots x_{i,b_i}
    $$
in the polynomial ring $K[x_{i,j}:1\le i\le n,1\le j\le b_i]$.  The \textit{polarization} of a monomial ideal $I$ is the squarefree monomial ideal $I^\wp$ of $S^\wp$ where \[S^\wp=K[x_{i,j}:1\le i\le n,1\le j\le\deg_{x_i}(I)].\] Therefore, its minimal generating set is given by $G(I^\wp)=\{u^\wp:u\in G(I)\}$.

	\begin{Proposition}\label{polarization proposition}
		Let $I\subset S$ be a monomial ideal. Then, the following hold.
		\begin{enumerate}[label=\textup{(\alph*)}]
                \item $\nu(I)=\nu(I^\wp)$.
                \item $(I^{[k]})^\wp=(I^\wp)^{[k]}$ for all $1\le k\le\nu(I)$.
			\item $g_I=g_{I^\wp}$.
			\item $\depth(S/I^{[k]})=\depth(S^\wp/(I^{\wp})^{[k]})-|{\bf deg}(I)|+n$, for all $1\le k\le\nu(I)$.
		\end{enumerate}
	\end{Proposition}
    \begin{proof}
    For any monomials $u, v\in S$, we have $\supp(u)\cap \supp(v)=\emptyset$ if and only if $\supp(u^\wp)\cap \supp(v^\wp)=\emptyset$. Therefore, $\nu(I)=\nu(I^\wp)$. Moreover, the equality $u_1^\wp\cdots u_k^\wp=(u_1\cdots u_k)^\wp$ holds whenever the monomials $u_1,\dots,u_k$ are in pairwise disjoint sets of variables, and thus (b) follows.
        
        By \cite[Corollary 1.6.3(d)]{HHBook2011} and equation in (b), it follows that
        $$
        \pd(S/I^{[k]})\ =\ \pd(S^\wp/(I^{[k]})^\wp)\ =\ \pd(S^\wp/(I^\wp)^{[k]}).
        $$
        Taking into account that $S^\wp$ is a polynomial ring in $|{\bf deg}(I)|$ variables, applying the Auslander--Buchsbaum formula we get
        $$
        \depth(S/I^{[k]})+|{\bf deg}(I)|-n=\depth(S^\wp/(I^\wp)^{[k]})
        $$
        which proves (d). Since $\textup{indeg}(I^{[k]})=\textup{indeg}((I^\wp)^{[k]})$, subtracting $\textup{indeg}(I^{[k]})-1$ from both sides of the above equation, we obtain
        \[g_{I}(k)\ =\ g_{I^\wp}(k), \ \ \textup{for all}\ \ 1\le k\le\nu(I).\]
    \end{proof}
    
In \cite[Corollary~3.5]{EHHM2022b} it was proved that $g_{I(G)}(\nu(G))=0$ for any fully supported edge ideal $I(G)$. By part (c) of the above proposition, we extend this to all quadratic monomial ideals.
 
	\begin{Corollary}\label{corollary:normalized depth zero}
		Let $I\subset S$ be a monomial ideal generated in degree two. Then $g_I(\nu(I))=0$.
	\end{Corollary}


\subsection*{Highest nonvanishing matching power of a quadratic monomial ideal}
Let $G$ be a finite simple graph with vertex set $V(G)=[n]=\{1,2,\dots,n\}$ and edge set $E(G)$. The \textit{edge ideal} of $G$ is the ideal $I(G)=(x_ix_j:\{i,j\}\in E(G))$ of $S=K[x_1,\dots,x_n]$. A \textit{matching} of $G$ is a set of edges of $G$ which are pairwise disjoint. If $M$ is a matching, then we denote by $V(M)$ the set of vertices $\bigcup_{e\in M}e$. We denote by $\nu(G)$ the \textit{matching number} of $G$ which is the maximum size of a matching of $G$. Then one can verify that $\nu(I(G))=\nu(G)$. Moreover, the generators of $I(G)^{[k]}$ correspond to matchings of $G$ of size $k$. This justifies the choice to name $I^{[k]}$ the $k$th matching power of $I$.

 A monomial ideal $I\subset S$ generated in a single degree is called \textit{polymatroidal} if its minimal generators correspond to the bases of a polymatroid. That is, the  \textit{exchange property} holds: for all $u,v\in G(I)$ and all $i$ with $\deg_{x_i}(u)>\deg_{x_i}(v)$ there exists $j$ such that $\deg_{x_j}(u)<\deg_{x_j}(v)$ and $x_j(u/x_i)\in G(I)$. A squarefree polymatroidal ideal is called \textit{matroidal}.
 
    A polymatroidal ideal has linear quotients with respect to the lexicographic order induced by any ordering of the variables. Indeed, a polymatroidal ideal is weakly polymatroidal and the above claim follows from \cite[Proof of Theorem~12.7.2]{HHBook2011}. Moreover, since polymatroidal ideals have linear quotients, they have linear resolutions as well \cite[Proposition~8.2.1]{HHBook2011}.

    Due to a well-known result of Edmonds and Fulkerson (see Theorem 1 on page 246, \cite{W}), the ideal $I(G)^{[\nu(G)]}$ is matroidal for any graph $G$. In particular, this implies that $I(G)^{[\nu(G)]}$ has linear quotients, which was independently proved by  Bigdeli et al. in \cite[Theorem 4.1]{BHZN18}. We will extend this to the matching power of any quadratic monomial ideal.

   If $I$ is a polymatroidal ideal, then $I^\wp$ is not necessarily polymatroidal. For instance, the ideal $I=(x_1^2, x_1x_2, x_2^2)$ is polymatroidal but $I^\wp$ is not. On the other hand, we have  

    \begin{Lemma}\label{Lemma:IwpPolym=>IPolym}
     Let $I\subset S$ be a monomial ideal. If $I^\wp$ is polymatroidal, then so is $I$.
\end{Lemma}
\begin{proof}
    Let $u,v\in G(I)$ with $p=\deg_{x_i}(u) >\deg_{x_i}(v)$. Then $x_{i,p}$ divides $u^\wp$ but not $v^\wp$. In fact,
\[\deg_{x_{i,p}}(u^\wp)=1>0=\deg_{x_{i,p}}(v^\wp).\]
Since $I^\wp$ is polymatroidal, there exists $x_{j,k}$ with $j\neq i$ such that 
\[\deg_{x_{j,k}}(v^\wp)=1>0=\deg_{x_{j,k}}(u^\wp)\]
and $x_{j,k}(u^\wp/x_{i,p})\in G(I^\wp)$. This implies $\deg_{x_j}(u)=k-1$ and $\deg_{x_j}(v)\geq k$. Then
\[(x_ju/x_i)^\wp = x_{j,k}(u^\wp/x_{i,p}) \in G(I^\wp)\]
and thus $x_ju/x_i\in G(I)$.
\end{proof}

\begin{Theorem}\label{Thm:I^[nu(I)]Polymatroidal}
		Let $I\subset S$ be a monomial ideal generated in degree two. Then $I^{[\nu(I)]}$ is a polymatroidal ideal.
	\end{Theorem}	
 
	\begin{proof}
		By Proposition~\ref{polarization proposition}(a) we have $k=\nu(I)=\nu(I^\wp)$. By Proposition~\ref{polarization proposition}(b), $(I^{[k]})^\wp=(I^\wp)^{[k]}$. Moreover, since $I^\wp$ is the edge ideal of some graph, it is known that $(I^\wp)^{[k]}$ is polymatroidal. Then Lemma~\ref{Lemma:IwpPolym=>IPolym} implies that $I^{[k]}$ is polymatroidal as well.
	\end{proof}
	
\begin{Corollary}
		Let $I\subset S$ be a monomial ideal generated in degree two. Then $\reg(I^{[\nu(I)]})=2\nu(I)$.
\end{Corollary}	
\begin{proof}
    Since polymatroidal ideals have linear resolutions and $I^{[\nu(I)]}$ is generated in degree $2\nu(I)$, the regularity formula follows.
\end{proof}
The above result and Corollary~\ref{corollary:normalized depth zero} are no longer valid for monomial ideals generated in a single degree bigger than two.  For instance, for the cubic ideal \[I=(x_1x_2^2, x_2x_3^2, x_3x_4^2, x_4x_1^2)\] of $S=K[x_1,\dots,x_4]$ we have $\nu(I)=2$ but $I^{[2]}$ does not have a linear resolution and $g_I(2)=1\neq 0$.

	
	\section{Edge ideals of weighted oriented graphs}\label{sec:2-EreyFic}
	
        In this section, we focus our attention on matching powers of edge ideals of weighted oriented graphs. The interest in these ideals stemmed from their relevance in coding theory, in particular in the study of Reed-Muller type codes \cite{MPV}. Recently, these ideals have been the subject of many research papers in combinatorial commutative algebra, e.g. \cite{BCDMS, BDS23, CK, HLMRV, KBLO, PRT}. Hereafter, by a graph $G$ we mean a finite simple undirected graph without isolated vertices.

        A (\textit{vertex})-\textit{weighted oriented graph} $\mathcal{D}=(V(\mathcal{D}), E(\mathcal{D}), w)$ consists of an underlying graph $G$ on which each edge is given an orientation and it is equipped with a \textit{weight function} $w:V(G)\rightarrow\mathbb{Z}_{\ge1}$. The \textit{weight} $w(i)$ of a vertex $i$ is denoted by
        $w_i$. The directed edges of $\mathcal{D}$ are denoted by pairs $(i,j)\in E(\mathcal{D})$ to reflect the orientation, hence $(i,j)$ represents an edge directed from $i$ to $j$. The \textit{edge ideal} of $\mathcal{D}$ is defined as the ideal
        $$
        I(\mathcal{D})\ =\ (x_ix_j^{w_j}\ :\ (i,j)\in E(\mathcal{D}))
        $$
        of the polynomial ring $S=K[x_i:i\in V(G)]$. If $w_i=1$ for all $i\in V(G)$, then $I(\mathcal{D})=I(G)$ is the usual edge ideal of $G$. 
        \begin{Remark}\label{remark: assumption on sources}
        \rm If $i\in V(G)$ is a \textit{source}, that is a vertex such that $(j,i)\notin E(\mathcal{D})$ for all $j$, then $\deg_{x_i}(I(\mathcal{D}))=1$. Therefore, hereafter we assume that $w_i=1$ for all sources $i\in V(G)$.
        \end{Remark}

	Firstly, we establish the homological comparison between the matching powers $I(\mathcal{D})^{[k]}$ and $I(G)^{[k]}$, where $G$ is the underlying graph of $\mathcal{D}$. The assumption in Remark~\ref{remark: assumption on sources} is crucial for the statement (e) of Theorem~\ref{Thm:comparison}. 
	\begin{Theorem}\label{Thm:comparison}
		Let $\mathcal{D}$ be a weighted oriented graph with underlying graph $G$. Then, the following statements hold.
		\begin{enumerate}
			\item[\textup{(a)}] $\nu(I(\mathcal{D}))=\nu(I(G))=\nu(G)$.
                \item[\textup{(b)}] $\pd(I(G)^{[k]})\le\pd(I(\mathcal{D})^{[k]})$, for all $1\le k\le\nu(G)$.
                \item[\textup{(c)}] $\reg(I(G)^{[k]})\le\reg(I(\mathcal{D})^{[k]})$, for all $1\le k\le\nu(G)$.
			\item[\textup{(d)}] $\beta_i(I(G)^{[k]})\le\beta_i(I(\mathcal{D})^{[k]})$, for all $1\le k\le\nu(G)$ and $i$.
			\item[\textup{(e)}] $g_{I(\mathcal{D})}(k)\le g_{I(G)}(k)+\sum\limits_{i\in V(G)}w_i-|V(G)|$, for all $1\le k\le\nu(G)$.
		\end{enumerate}
	\end{Theorem}

    For the proof we recall a few basic facts. Let $I\subset S$ be a monomial ideal.
    \begin{enumerate}
    	\item[(i)] We have $\beta_{i,j}(I)=\beta_{i,j}(I^\wp)$ for all $i$ and $j$ \cite[Corollary 1.6.3]{HHBook2011}.
    	\item[(ii)] For a monomial $u\in S$, we set $\sqrt{u}=\prod_{i\in\supp(u)}x_i$. If $G(I)=\{u_1,\dots,u_m\}$, then  \cite[Proposition 1.2.4]{HHBook2011} gives
    	$$\sqrt{I}=(\sqrt{u_1},\dots,\sqrt{u_m}).$$
    	\item[(iii)] Let $P$ be a monomial prime ideal of $S$. Let $S(P)$ be the polynomial ring in the variables which generate $P$. The \textit{monomial localization} of $I$ at $P$ is the monomial ideal $I(P)$ of $S(P)$ which is obtained from $I$ by the substitution $x_i\mapsto1$ for all $x_i\notin P$. The monomial localization can also be described as the saturation $I:(\prod_{x_i\notin P}x_i)^\infty$.
     
        If $\FF$ is the minimal (multi)graded free $S$-resolution of $I$, one can construct, starting from $\FF$, a possibly non-minimal (multi)graded free $S$-resolution of $I(P)$ \cite[Lemma 1.12]{HMRZ021a}. It follows from this construction that $\beta_{i}(I(P))\le\beta_{i}(I)$ for all $i$. Moreover, $\pd(I(P))\le\pd(I)$ and $\reg(I(P))\le\reg(I)$.
    \end{enumerate}

    \begin{proof}
    	Statement (a) is clear. To prove (b), (c) and (d), set $J=I(\mathcal{D})^{[k]}$.
     Assume that $I(\mathcal{D})$ is a fully supported ideal of $S=K[x_1,\dots, x_n]$. Let $P=(x_{1,1},\dots,x_{n,1})$. Identifying $x_{i,1}$ with $x_i$ for all $i$, by applying (ii), $J^\wp(P)$ can be identified with $\sqrt{J}$. Then by (i) and (iii) we obtain
     \[\beta_{i}(\sqrt{J})=\beta_{i}(J^\wp(P))\leq \beta_{i}(J^\wp)=\beta_{i}(J)\]
     for all $i$. To complete the proof, we will show that $\sqrt{J}=I(G)^{[k]}$. For this aim, let $v\in G(J)$. Then $v=(x_{i_1}x_{j_1}^{w_{j_1}})\cdots(x_{i_k}x_{j_k}^{w_{j_k}})$ with $(i_1,j_1),\dots,(i_k,j_k)\in E(\mathcal{D})$ and the corresponding undirected edges form a $k$-matching of $G$. Thus $\sqrt{v}=(x_{i_1}x_{j_1})\cdots(x_{i_k}x_{j_k})\in I(G)^{[k]}$ and consequently $\sqrt{J}\subseteq I(G)^{[k]}$. Conversely, let $u=(x_{i_1}x_{j_1})\cdots(x_{i_k}x_{j_k})\in G(I(G)^{[k]})$ with $\{\{i_1,j_1\},\dots,\{i_k,j_k\}\}$ a $k$-matching of $G$. Then $(i_1,j_1),\dots,(i_k,j_k)\in E(\mathcal{D})$ up to relabelling. So $v=(x_{i_1}x_{j_1}^{w_{j_1}})\cdots(x_{i_k}x_{j_k}^{w_{j_k}})\in J$ and $\sqrt{v}=u\in\sqrt{J}$. This shows that $I(G)^{[k]}\subseteq\sqrt{J}$. Equality follows.
    	
    	It remains to prove (e). Let $L$ be a monomial ideal of $S$. By the Auslander--Buchsbaum formula we have $\depth(S/L)=n-1-\pd(L)$. Hence, for all $1\le k\le\nu(L)$ we can rewrite $g_L(k)$ as
    	$$
    	g_L(k)=|{\bf deg}(L)|-\pd(L^{[k]})-\textup{indeg}(L^{[k]}).
    	$$
    	By (b) we have $\pd(I(G)^{[k]})\le\pd(I(\mathcal{D})^{[k]})$ for all $k$. It is clear that $|{\bf deg}(I(G))|=n$ and $\textup{indeg}(I(G)^{[k]})=2k\le\textup{indeg}(I(\mathcal{D})^{[k]})$ for all $1\le k\le\nu(G)$. Therefore,
    	\begin{eqnarray*}
    		g_{I(\mathcal{D})}(k)&=&|{\bf deg}(I(\mathcal{D}))|-\pd(I(\mathcal{D})^{[k]})-\textup{indeg}(I(\mathcal{D})^{[k]})\\
    		&\le&|{\bf deg}(I(\mathcal{D}))|-\pd(I(G)^{[k]})-\textup{indeg}(I(G)^{[k]})\\
    		&=& n-\pd(I(G)^{[k]})-\textup{indeg}(I(G)^{[k]})+|{\bf deg}(I(\mathcal{D}))|-n\\
    		&=& g_{I(G)}(k)+|{\bf deg}(I(\mathcal{D}))|-n.
    	\end{eqnarray*}
        Since $\deg_{x_i}(I(\mathcal{D}))=w_i$ for all $i$, we have $|{\bf deg}(I(\mathcal{D}))|=\sum_{i=1}^nw_i$, as wanted.
    \end{proof}
	
	The inequalities in (b), (c), (d) and (e) need not to be equalities as one can see in the next example.
	\begin{Example}
		\rm Let $\mathcal{D}$ be the oriented 4-cycle with all vertices having weight 2 and with edge set $E(\mathcal{D})=\{(a,b),(b,c),(c,d),(d,a)\}$. Then $I(G)^{[2]}=(abcd)$, while $I(\mathcal{D})^{[2]}=(ab^2cd^2,a^2bc^2d)$. By using \textit{Macaulay2} \cite{GDS} and the package \cite{FPack2}, we checked that $\pd(I(G)^{[2]})=1<2=\pd(I(\mathcal{D})^{[2]})$, $\reg(I(G)^{[2]})=4<7=\reg(I(\mathcal{D})^{[2]})$, $\beta_1(I(G)^{[2]})=0<1=\beta_1(I(\mathcal{D})^{[2]})$, and $g_{I(G)}(2)=1< 5=g_{I(\mathcal{D})}(2)+\sum_{i=1}^4w_i-4$.
	\end{Example}


        Hereafter, we concentrate our attention on edge ideals of vertex-weighted oriented graphs. Let $\mathcal{D}'$ and $\mathcal{D}$ be weighted oriented graphs with underlying graphs $G'$ and $G$ respectively. We say $\mathcal{D}'$ is a \textit{weighted oriented subgraph} of $\mathcal{D}$ if the vertex and edge sets of $\mathcal{D}'$ are contained in respectively those of $\mathcal{D}$ and the weight functions coincide on $V(\mathcal{D}')$. A weighted oriented subgraph $\mathcal{D}'$ of $\mathcal{D}$ is called \textit{induced weighted oriented subgraph} of $\mathcal{D}$ if $G'$ is an induced subgraph of $G$.
        
        Firstly, we turn to the problem of bounding the regularity of matching powers of edge ideals. We begin with the so-called Restriction Lemma.
        
	\begin{Lemma}\label{lem:induced subgraph}
		Let $\mathcal{D}'$ be an induced weighted oriented subgraph of $\mathcal{D}$. Then 
		\begin{enumerate}
			\item[\textup{(a)}] $\beta_{i,{\bf a}}(I(\mathcal{D}')^{[k]}) \leq \beta_{i,{\bf a}}(I(\mathcal{D})^{[k]})$ for all $i$ and ${\bf a}\in \mathbb{Z}^n$.
			\item[\textup{(b)}] $\reg(I(\mathcal{D}')^{[k]}) \leq \reg(I(\mathcal{D})^{[k]})$.
		\end{enumerate}
	\end{Lemma}
	\begin{proof}
		It follows from \cite[Lemma 1.2]{EHHM2022a}.
	\end{proof}

	Let $\im(G)$ denote the \textit{induced matching number} of $G$. For any weighted oriented graph $\mathcal{D}$ with underlying graph $G$, let $\wim(\mathcal{D})$ denote the \textit{weighted induced matching number} of $\mathcal{D}$. That is,
    \begin{align*}
        \wim(\mathcal{D})=\max\big\{\!\sum_{i=1}^{m}w(y_i)\ :\ \ &\{\{x_1,y_1\},\dots,\{x_m,y_m\}\}\ \text{is an}\\&\ \text{induced matching of}\ G,\ \text{and}\ (x_i,y_i)\in E(\mathcal{D})\big\}.
    \end{align*}
	Notice that if $w_i=1$ for every $i\in V(\mathcal{D})$, then $\wim(\mathcal{D})=\im(G)$. Otherwise, we have the inequality $\wim(\mathcal{D})\geq \im(G)$. We extend the regularity lower bound given in \cite[Theorem~3.8]{BDS23} as follows.
	\begin{Proposition}\label{prop:regularity lower bound}
		Let $\mathcal{D}$ be a weighted oriented graph with underlying graph $G$. Then 
		\[\reg(I(\mathcal{D})^{[k]})\geq \wim(\mathcal{D})+k\]
		for all $1\leq k \leq \im(G)$.
	\end{Proposition}
	\begin{proof}
		The proof is similar to \cite[Theorem~2.1]{EHHM2022a}. We include the details for the sake of completeness. Let $\{\{x_1,y_1\},\dots ,\{x_r,y_r\}\}$ be an induced matching. Suppose that $(x_i,y_i)\in E(\mathcal{D})$ with $w(y_i)=t_i$ and $\sum_{i=1}^{r}t_i=\wim(\mathcal{D})$. Let $\mathcal{D}'$ be the induced weighted oriented subgraph of $\mathcal{D}$ on the vertices $x_1,\dots,x_r,y_1,\dots,y_r$. Then by Lemma~\ref{lem:induced subgraph} it suffices to show that
		\[\reg(I(\mathcal{D}')^{[k]})\geq \wim(\mathcal{D})+k.\]
		To this end, we set $I=I(\mathcal{D}')$ and we claim that
		\[\beta_{r-k,\wim(\mathcal{D})+r}(I^{[k]})\neq 0.\]
		Let $J = (z_1, \dots, z_r)$, where $z_1, \dots , z_r$ are new variables. Then $J^{[k]}$ is a squarefree strongly stable ideal in the polynomial ring $R = K[z_1, \dots, z_r]$. It was proved in \cite[Theorem~2.1]{EHHM2022a} that $\beta_{r-k, r}(J^{[k]})\neq 0$. 
		
		Define the map $\phi: R \to S = K[x_1, \dots , x_r, y_1, \dots , y_r]$ by $z_i \mapsto x_iy_i^{t_i}$ for $i = 1, \dots , r$. Since $x_1y_1^{t_1}, \dots , x_ry_r^{t_r}$ is a regular sequence on $S$, the $K$-algebra homomorphism $\phi$ is flat. If $\mathbb{F}$ is the minimal free resolution of $J^{[k]}$ over $R$, then $\mathbb{G}:\mathbb{F}\otimes_R S$ is the minimal free resolution of $I^{[k]}$ over $S$. It follows that
		\[\beta_{i,(a_1,\dots ,a_r)}(J^{[k]})=\beta_{i,(a_1,\dots ,a_r,t_1a_1,\dots ,t_ra_r)}(I^{[k]})\]
		for any $i$ and $(a_1,\dots ,a_r)\in \mathbb{Z}^r$. Then,
		\[0\neq\beta_{r-k,r}(J^{[k]})=\beta_{r-k,(1,\dots ,1)}(J^{[k]})=\beta_{r-k,(1,\dots ,1,t_1,\dots ,t_r)}(I^{[k]})\]
		and $\beta_{r-k,\wim(\mathcal{D})+r}(I^{[k]})\neq 0$ as desired.
	\end{proof}

    

 We close this section by providing a lower bound for the projective dimension of matching powers of edge ideals. Let $P_n$ be the \textit{path of length} $n$. That is, $V(P_n)=[n]$ and $E(P_n)=\{\{1,2\},\{2,3\},\dots,\{n-1,n\}\}$. We denote by $\mathcal{P}_n$ a weighted oriented path of length $n$, that is, a weighted oriented graph whose underlying graph is $P_n$. It is well--known that $\nu(P_n)=\lfloor\frac{n}{2}\rfloor$.

 For a weighted oriented graph $\mathcal{D}$ with underlying graph $G$, we denote by $\ell(\mathcal{D})$ the maximal length of an induced path of $G$. 
 \begin{Proposition}\label{prop:projective dimension lower bound}
     Let $\mathcal{D}$ be a weighted oriented graph. Then $\nu(I(\mathcal{D}))\ge\lfloor\frac{\ell(\mathcal{D})}{2}\rfloor$ and
     $$
     \pd(I(\mathcal{D})^{[k]})\ \ge\ \begin{cases}
                \ell(\mathcal{D})-\lceil\frac{\ell(\mathcal{D})}{3}\rceil-k&\text{if}\ 1\le k\le\lceil\frac{\ell(\mathcal{D})}{3}\rceil,\\
                \ell(\mathcal{D})-2k&\text{if}\ \lceil\frac{\ell(\mathcal{D})}{3}\rceil+1\le k\le\lfloor\frac{\ell(\mathcal{D})}{2}\rfloor.
            \end{cases}
     $$
 \end{Proposition}
\begin{proof}
    Let $\ell=\ell(\mathcal{D})$. There exists a subset $W$ of $V(\mathcal{D})$ such that the induced subgraph of $\mathcal{D}$ on $W$ is a weighted oriented path $\mathcal{P}_\ell$. Theorem \ref{Thm:comparison}(b) combined with Lemma \ref{lem:induced subgraph} implies that $\pd(I(P_\ell)^{[k]})\le\pd(I(\mathcal{P}_\ell)^{[k]})\le\pd(I(\mathcal{D})^{[k]})$. It was shown in \cite[Theorem 2.10]{CFL} that
    $$
		g_{I(P_\ell)}(k)\ =\ \begin{cases}
		\lceil\frac{\ell}{3}\rceil-k&\text{if}\ 1\le k\le\lceil\frac{\ell}{3}\rceil,\\
		\hfil0&\text{if}\ \lceil\frac{\ell}{3}\rceil+1\le k\le\lfloor\frac{\ell}{2}\rfloor.
		\end{cases}
		$$
    For a squarefree monomial ideal $I\subset S$, we have $g_I(k)=n-\pd(I^{[k]})-\textup{indeg}(I^{[k]})$. Hence, the assertion follows from the above formula.
\end{proof}

Although we only considered weighted oriented graphs in this section, our methods can be useful to prove analogous results for matching powers of edge ideals of edge-weighted graphs. An \textit{edge-weighted graph} $G_w$ consists of a graph $G$ equipped with a \textit{weight function} $w:E(G)\rightarrow\mathbb{Z}_{\ge1}$. The \textit{edge ideal} of $G_w$ is defined as the ideal
       $$
        I(G_w) = ((x_ix_j)^{w(e)} :e=\{i,j\}\in E(G))
        $$
        of $S=K[x_i:i\in V(G)]$, see \cite{PS13}. Notice that if the weight of every edge is $1$, then the edge ideal of $G_w$ coincides with that of $G$.


 
 
\section{Linearly related matching powers}\label{sec:3-EreyFic}
Let $I\subset S$ be a graded ideal generated in a single degree. We say $I$ is {\em linearly related}, if the first syzygy module of $I$ is generated by linear relations. In this section, we want to discuss which matching powers of the edge ideal $I(\mathcal{D})$ of a vertex-weighted oriented graph $\mathcal{D}$ are linearly related.

Let $I$ be a monomial ideal of $S$ generated in degree $d$. Let $G_I$ denote the graph with vertex set $G(I)$ and edge set 
	\[
	E(G_I)=\{\{u,v\}: u,v\in G(I)~\text{with}~\deg (\lcm(u,v))=d+1 \}.
	\]

	For all $u, v\in G(I)$ let  $G^{(u,v)}_I$ be the induced subgraph of $G_I$ whose vertex set is
	\[
	V(G^{(u,v)}_I)=\{w\in G(I)\: \text{$w$ divides $\lcm(u, v)$}\}.
	\]

	\medskip
	The following theorem provides a criterion through the graphs defined above to determine if a monomial ideal is linearly related. 
	
	\begin{Theorem}\label{connected criterion}\cite[Corollary~2.2]{BHZN18} 
		Let $I$ be a monomial ideal generated in degree $d$. Then $I$ is linearly related if and only if for all $u,v\in G(I)$  there is a path in $G^{(u,v)}_I$ connecting $u$ and $v$.	
	\end{Theorem}

\begin{Lemma}\label{lem:generated in single degree}
	Let $I$ be a monomial ideal and let $1\leq k < \nu(I)$. Suppose that $I^{[k]}$ is generated in single degree. Then, there is an integer $d$ such that
	\begin{itemize}
		\item[\textup{(a)}] $I^{[k]}$ is generated in degree $dk$,
		\item[\textup{(b)}] $I^{[k+1]}$ is generated in degree $d(k+1)$, and
            \item[\textup{(c)}] if $u=u_1\dots  u_{k+1}\in G(I^{[k+1]})$, with each $u_i\in G(I)$ and $\gcd(u_i, u_j) =1$ for $i\neq j$, then $\deg(u_i)=d$ for each $i$.
	\end{itemize}

\end{Lemma}
\begin{proof}
	Let $u=u_1\cdots u_{k+1}\in G(I^{[k+1]})$ with each $u_i\in G(I)$ and $\gcd(u_i, u_j)=1$ for $i\neq j$. Observe that $u/u_\ell\in G(I^{[k]})$ for any $\ell=1,\dots ,k+1$. First, we show that $\deg(u_i)=\deg(u_j)$ for each $i\neq j$. Without loss of generality, assume for a contradiction that $\deg(u_1)\neq \deg(u_2)$. Then $u_2u_3\cdots u_{k+1}$ and $u_1u_3\cdots u_{k+1}$ are minimal monomial generators of $I^{[k]}$ of different degrees, which is a contradiction. It follows that $u_1,\dots ,u_k$ are all of the same degree, say $d$. Now, suppose that $v=v_1\cdots v_{k+1}\in G(I^{[k+1]})$ with each $v_i\in G(I)$ and $\gcd(v_i,v_j)=1$ for $i\neq j$. By the above argument, each $v_i$ is of the same degree, say $d'$. Then $u_1\cdots u_k$ is a minimal monomial generator of $I^{[k]}$ of degree $dk$ whereas $v_1\cdots v_k$ is a minimal monomial generator of $I^{[k]}$ of degree $d'k$. Therefore $d=d'$ and $u$ and $v$ have the same degree.
\end{proof}

In \cite[Theorem~3.1]{BHZN18} it was proved that $I(G)^s$ is linearly related for some $s\geq 1$ if and only if $I(G)^k$ is linearly related for all $k\geq 1$. Unlike the ordinary powers of edge ideals, not all squarefree powers of $I(G)$ are linearly related if some squarefree power is linearly related. On the other hand, it was proved in \cite[Theorem~3.1]{EHHM2022a} that if $I(G)^{[k]}$ is linearly related for some $k\geq 1$, then $I(G)^{[k+1]}$ is linearly related as well. We extend \cite[Theorem~3.1]{EHHM2022a} to  monomial ideals, under some additional assumptions.
	
	\begin{Theorem}[\textbf{A condition for consecutive linearly related powers}]\label{thm:linearly related consecutive powers}
		Let $I$ be a monomial ideal such that $|\supp(w)|=2$ for every $w\in G(I)$. Suppose that $I^{[k]}$ is linearly related for some $1\leq k <\nu(I)$. If $\supp(u)\neq \supp(v)$ for every $u,v\in G(I^{[k+1]})$ with $u\neq v$, then $I^{[k+1]}$ is linearly related.
	\end{Theorem}
	\begin{proof}
		Suppose that $\supp(u)\neq \supp(v)$ for every $u,v\in G(I^{[k+1]})$ with $u\neq v$. By the previous lemma, $I^{[k]}$ is generated in degree $dk$, and $I^{[k+1]}$ is generated in degree $d(k+1)$. Let $u,v \in G(I^{[k+1]})$ with $u\neq v$. By Theorem~\ref{connected criterion} and Lemma~\ref{lem:generated in single degree}, it suffices to find a path in $G_{I^{[k+1]}}^{(u,v)}$ connecting $u$ to $v$. Let $u=u_1\cdots u_{k+1}$ and let $v=v_1\cdots v_{k+1}$ where $u_i,v_i \in G(I)$ for each $i=1,\dots ,k+1$ and
		\[\supp(u_p)\cap \supp(u_q)=\emptyset=\supp(v_p)\cap \supp(v_q)\]
		for every distinct $p,q\in \{1,\dots, k+1\}$. By Lemma~\ref{lem:generated in single degree}, we have that $\deg(u_i)=\deg(v_i)=d$ for every $i=1,\dots ,k+1$. By the initial assumption, we may assume that there exists $\ell \in \supp(u)\setminus \supp(v)$. Without loss of generality, we may assume that $x_\ell$ divides $u_1$. Let $\supp(u_1)=\{\ell, m\}$. By definition of matching power, there exists at most one $j$ such that $x_m$ divides $v_j$. Again, without loss of generality, we may assume that $x_m$ does not divide $v_i$ for $i=2,\dots ,k+1$. Now, we have
		\[\supp(u_1)\cap \supp(v_p)=\emptyset\ \ \text{ for all }\ \ p=2,3,\dots ,k+1.\]
		
		Let $u'=u_2\dots u_{k+1}$ and $v'=v_2\dots v_{k+1}$. Since $u', v'\in G(I^{[k]})$ there exists a path $u'=z_0, z_1, z_2, \dots , z_t, v'=z_{t+1}$ in $G_{I^{[k]}}^{(u',v')}$ connecting $u'$ to $v'$. We claim that
		\[P: u, u_1z_1, u_1z_2, \dots ,u_1z_t, u_1v'\]
		is a path in $G_{I^{[k+1]}}^{(u,u_1v')}$. To prove the claim, we must show that
  \begin{itemize}
      \item[(i)] $u_1z_i\in G(I^{[k+1]})$ for all $i=1,\dots, t+1$,
      \item[(ii)] $u_1z_i$ divides $\lcm(u,u_1v')$ for all $i=1,\dots ,t$ and,
      \item[(iii)] $\deg(\lcm(u_1z_i, u_1z_{i+1}))=d(k+1)+1$ for all $i=0,\dots ,t$.
  \end{itemize}
  Since $\supp(u_1)\cap \supp(\lcm(u', v'))=\emptyset$, the monomial $u_1z_i$ belongs to $I^{[k+1]}$ for all $i=1,\dots ,t+1$. Moreover, since $u_1z_i$ is of degree $d(k+1)$, it follows that $u_1z_i \in G(I^{[k+1]})$, which proves (i). To see (ii) holds, observe that 
  \[\lcm(u,u_1v')=\lcm(u_1z_0, u_1z_{t+1})=u_1\lcm(z_0, z_{t+1}).\]
  Lastly, (iii) holds because for all $i=0,\dots ,t$ we have
  \[\deg(\lcm(u_1z_i, u_1z_{i+1}))=\deg(u_1)+\deg(\lcm(z_i, z_{i+1}))=d+ (dk+1).\]
  
  Now, let $w=u_1v_2\dots v_k$ and $w'=v_1v_2\dots v_k$. Since $w, w'\in G(I^{[k]})$ there exists a path $w, y_1, y_2, \dots , y_s, w'$ in $G_{I^{[k]}}^{(w,w')}$ connecting $w$ to $w'$. As before, we can then form a path $P'$
		\[P': wv_{k+1}, y_1v_{k+1}, y_2v_{k+1}, \dots , y_sv_{k+1}, w'v_{k+1}=v\]
		in $G_{I^{[k+1]}}^{(u_1v', v )}$. Connecting $P$ and $P'$ we get the required path, as $u_1v'=wv_{k+1}$.
	\end{proof}

A matching $M$ of $G$ is called a \textit{perfect matching} of $G$ if $V(M)=V(G)$. We say that a cycle (or a path) is even (respectively odd) if it has an even (respectively odd) number of edges. A connected graph $G$ is called a \textit{cactus graph} if every edge of $G$ belongs to at most one cycle of $G$.

For the rest of this section, we will be interested in graphs which satisfy the property that every subgraph of them has at most one perfect matching. One can characterize the connected components of these graphs as special cactus graphs.

\begin{Lemma}[\textbf{Description of graphs studied in this section}]\label{lem: special cactus}
    Let $G$ be a graph without isolated vertices. Then the following statements are equivalent.
    \begin{itemize}
        \item[\textup{(a)}] Every subgraph of $G$ has at most one perfect matching.
        \item[\textup{(b)}] $G$ has no even cycles.
        \item[\textup{(c)}] Every block of $G$ is an edge or an odd cycle, i.e., every connected component of $G$ is a cactus graph whose blocks are edges or odd cycles.
    \end{itemize}
\end{Lemma}
\begin{proof}
   First, observe that every even cycle has two perfect matchings. On the other hand, suppose that $H$ is a subgraph of $G$ with two perfect matchings $M_1$ and $M_2$. Consider the graph $N$ whose edge set is the symmetric difference $M_1\triangle M_2$ and whose vertex set is $V(M_1\triangle M_2)$. Then every vertex of $N$ belongs to exactly two edges of $N$, one from $M_1\setminus M_2$ and the other from $M_2\setminus M_1$. Therefore, $N$ consists of a disjoint union of even cycles, which shows the equivalence of (a) and (b). It is known that (b) $\implies$ (c), see Exercises 4.1.31 and 4.2.18 in \cite{WEST}. Lastly, (c) $ \implies$ (b) because every even cycle, being a connected graph with no cut vertex, must be contained in a block of $G$.
\end{proof}

We will now observe that the assumption of the previous theorem is satisfied for edge ideals of weighted oriented graphs whose underlying graphs are as in Lemma~\ref{lem: special cactus}. Hereafter, to simplify the notation, we identify each vertex $i\in V(\mathcal{D})$ with the variable $x_i$. Hence, we will often write $x_i$ to denote $i$.
 \begin{Lemma}\label{lem: distinct support}
     Let $\mathcal{D}$ be a weighted oriented graph with underlying graph $G$. Suppose that every subgraph of $G$ has at most one perfect matching. Let $1\leq k\leq \nu(G)$ and $u,v\in G(I(\mathcal{D})^{[k]})$. If $\supp (u)= \supp (v)$, then $u=v$. 
 \end{Lemma}
 \begin{proof}
 Let $u=x_1y_1^{w(y_1)}\dots x_ky_k^{w(y_k)}$ where $(x_i,y_i)\in E(\mathcal{D})$ for each $i$ and $M_1=\{\{x_i,y_i\} : i=1,\dots ,k\}$ is a matching in $G$. Let $v=z_1t_1^{w(t_1)}\dots z_kt_k^{w(t_k)}$ where $(z_i,t_i)\in E(\mathcal{D})$ for each $i$ and $M_2=\{\{z_i,t_i\} : i=1,\dots ,k\}$ is a matching in $G$. Suppose that $\supp (u) =\supp (v)$. Then we can set $W:=V(M_1)=V(M_2)$. Since the induced subgraph of $G$ on $W$ has at most one perfect matching, it follows that $M_1=M_2$ and therefore $u=v$. 
 \end{proof}\medskip


If $I(\mathcal{D})\ne I(G)$ and $G$ is as in Lemma~\ref{lem: special cactus}, then we will see in Theorem~\ref{thm: linearly related only in matching power} that the only linearly related matching power is the last one. Before we can prove Theorem \ref{thm: linearly related only in matching power}, we need some preliminary lemmas. Hereafter, with abuse of notation, for a monomial $u$, we denote by $\supp(u)$ also the set of variables dividing $u$.

\begin{Lemma}\label{lem: if deg lcm(u,v)=d+1}
    Let $\mathcal{D}$ be a weighted oriented graph and let $1\leq k\leq \nu(I(\mathcal{D}))$.
    \begin{enumerate}
        \item[\textup{(a)}] Suppose that every subgraph of the underlying graph $G$ of $\mathcal{D}$ has at most one perfect matching. Then, $u\in G(I(\mathcal{D}^{[k]}))$ if and only if $u=x_1y_1^{w(y_1)}\dots x_ky_k^{w(y_k)}$ for some $(x_i,y_i)\in E(\mathcal{D})$ with $\{\{x_i,y_i\} : i=1,\dots ,k\}$ a matching in $G$. 
        \item[\textup{(b)}] Let $u, v\in G(I(\mathcal{D})^{[k]})$ such that $\supp(u)\neq \supp(v)$ and \[\deg (\lcm(u,v))=\deg (u)+1=\deg(v)+1.\]  Then there exist variables $z_1\notin \supp(u)$, $z_2\notin\supp(v)$ such that $v=uz_1/z_2$, $\deg_{z_1}(v)=1$ and $\deg_{z_2}(u)=1$.
    \end{enumerate}
\end{Lemma}
\begin{proof}
(a) The ``only if" side of the statement is by definition of matching power. The ``if" side follows from Lemma~\ref{lem: distinct support} and the fact that every minimal monomial generator of $I(\mathcal{D})^{[k]}$ has a support of size $2k$.
 
    (b) Since both $u$ and $v$ have support of size $2k$ and $\supp(u)\neq \supp(v)$, there exists a variable $z_1\in \supp(v)\setminus \supp(u)$ and $z_2\in \supp(u)\setminus \supp(v)$. Since $\deg (\lcm(u,v))=\deg (u)+1$, we get $\supp(v)\setminus \supp(u)=\{z_1\}$ and $\deg_{z_1}(v)=1$. Similarly, since $\deg (\lcm(u,v))=\deg (v)+1$, we get $\supp(u)\setminus \supp(v)=\{z_2\}$ and $\deg_{z_2}(u)=1$. Then for every $t\in \supp(u)\cap \supp(v)$, we get $\deg_t(u)=\deg_t(v)$ and the result follows.
\end{proof}
\begin{Lemma}\label{lem:degree of variable bigger than 1}
    Let $\mathcal{D}$ be a weighted oriented graph with underlying graph $G$. Suppose that every subgraph of $G$ has at most one perfect matching. Suppose that $I(\mathcal{D})^{[k]}$ is linearly related. Let $u\in G(I(\mathcal{D})^{[k]})$ and let $x$ be a variable such that $\deg_{x}(u)=r>1$. Then $\deg_{x}(v)=r$ for every $v\in G(I(\mathcal{D})^{[k]})$.  
\end{Lemma}
\begin{proof}
    Let $u\neq v$. By Theorem~\ref{connected criterion} there is a path $u_0=u, u_1, u_2, \dots ,u_s=v$ in the graph $H:=G^{(u,v)}_{I(\mathcal{D})^{[k]}}$. Since $\{u_0,u_1\}\in E(H)$, by Lemma~\ref{lem: distinct support} and Lemma~\ref{lem: if deg lcm(u,v)=d+1}(b) it follows that $\deg_{x}(u_1)=r$. Similarly, since $\{u_1,u_2\}\in E(H)$ it follows that $\deg_{x}(u_2)=r$. Continuing this way, we obtain $\deg_{x}(u_s)=r$.
\end{proof}

\begin{Theorem}[\textbf{The only linearly related power is the highest one}]\label{thm: linearly related only in matching power}
	Let $G$ be the underlying graph of $\mathcal{D}$. Suppose that every subgraph of $G$ has at most one perfect matching, and that $I(\mathcal{D})\neq I(G)$. Let $1\leq k\leq \nu(G)$. If $I(\mathcal{D})^{[k]}$ is linearly related, then $k=\nu(G)$.
\end{Theorem}
\begin{proof}
	Assume for a contradiction that $I(\mathcal{D})^{[k]}$ is linearly related but $k<\nu(G)$. Let $M=\{\{a_i,b_i\} : i=1,\dots ,k+1\}$ be a matching with $(a_i,b_i)\in E(\mathcal{D})$. We claim that all the $b_i$s have the same weight. To see this, we let $u_i=a_ib_i^{w(b_i)}$ for each $i$ and $z=u_1\dots u_{k+1}$. Then Lemma~\ref{lem: if deg lcm(u,v)=d+1}(a) implies that $z/u_i \in G(I(\mathcal{D})^{[k]})$ for each $i=1,\dots,k+1$. Since $I(\mathcal{D})^{[k]}$ is generated in single degree, it follows that there is a positive integer $q$ such that $w(b_i)=q$ for all $i=1,\dots ,k+1$.

Since $I(\mathcal{D})\neq I(G)$ there is an edge $(c,d)\in E(\mathcal{D})$ with $w(d)=r>1$. We will show that $r=q$. Without loss of generality, we may assume that \[\{c,d\}\cap V(M)\subseteq \{a_1,b_1, a_2,b_2\}.\]
 
 Then $\{\{c,d\}, \{a_3,b_3\}, \dots ,\{a_{k+1},b_{k+1}\}\}$ is a matching. On the other hand, by Lemma~\ref{lem: if deg lcm(u,v)=d+1}(a) both $cd^ru_3\dots u_{k+1}$ and $u_2\dots u_{k+1}$ are minimal generators of $I(\mathcal{\mathcal{D}})^{[k]}$. Since $I(\mathcal{D})^{[k]}$ is generated in single degree, it follows that $r=q>1$.
	
	Let $w_1=u_1\dots u_k$ and $w_2=u_2\dots u_{k+1}$. Using Lemma~\ref{lem:degree of variable bigger than 1} and comparing the $b_1$-degrees of $w_1$ and $w_2$ we obtain a contradiction.
\end{proof}
The next example shows that we can not drop the hypothesis that every subgraph of $G$ has at most one perfect matching.
\begin{Example}
	\rm Let $\mathcal{D}$ be the oriented graph on vertex set $[6]$, with weights $w(1)=2$ and $w(i)=1$ for $i\neq 1$, and with edge set $$E(\mathcal{D})=\{(2,1),(1,3),(1,4),(1,5),(1,6)\}\cup\{(i,j):2\le i<j\le 6\}.$$
	Then, $G$ has several perfect matchings, and
	$$
	I(\mathcal{D})=(x_1^2x_2,x_1x_3,x_1x_4,x_1x_5,x_1x_6,x_2x_3,x_2x_4,x_2x_5,x_2x_6,\dots,x_4x_5,x_4x_6,x_5x_6).
	$$
	We have $\nu(I(\mathcal{D}))=3$. However $I(\mathcal{D})^{[2]}=I(G)^{[2]}$ and $I(\mathcal{D})^{[3]}=I(G)^{[3]}$ are linearly related, indeed they even have a linear resolution.
\end{Example}

We can now characterize when $I(\mathcal{D})^{[k]}$ has a linear resolution or is linearly related provided that the underlying graph $G$ is as in Lemma~\ref{lem: special cactus}.


\begin{Theorem}[\textbf{Characterization of linear resolutions}]\label{Thm:I(D)linRel}
    Let $G$ be the underlying graph of $\mathcal{D}$. Suppose that every subgraph of $G$ has at most one perfect matching. Suppose that $I(\mathcal{D})\neq I(G)$ and $1\leq k \leq \nu(G)$. Then the following statements are equivalent.
    \begin{enumerate}
        \item[\textup{(a)}] $I(\mathcal{D})^{[k]}$ is linearly related.
        \item[\textup{(b)}] $I(\mathcal{D})^{[k]}$ is polymatroidal.
        \item[\textup{(c)}] $I(\mathcal{D})^{[k]}$ has a linear resolution.
    \end{enumerate}
\end{Theorem}
\begin{proof}
A polymatroidal ideal has linear quotients \cite[Theorem~12.6.2]{HHBook2011} and therefore it has a linear resolution \cite[Proposition~8.2.1]{HHBook2011}.
    We will only show that (a) $\implies$ (b) because (b) $\implies$ (c) $\implies$ (a) is known.

    Suppose that $I(\mathcal{D})^{[k]}$ is linearly related. Then $k=\nu(G)$ by Theorem~\ref{thm: linearly related only in matching power}. Let $u,v\in G(I(\mathcal{D})^{[k]})$. Let $M=\{e_1,\dots,e_k\}$ and $N=\{f_1,\dots,f_k\}$ be the underlying matchings (of undirected edges) for respectively $u$ and $v$.  Let $M_{e_i}$ and $N_{f_i}$ be the monomial factors of $u$ and $v$ (respectively) corresponding to $e_i$ and $f_i$. That is, $u=M_{e_1}\dots M_{e_k}$ and $v=N_{f_1}\dots N_{f_k}$ with $\supp(M_{e_i})=e_i$ and $\supp(N_{f_i})=f_i$.

    As in the proof of \cite[Theorem 1 on page 246]{W}, let $H$ be the graph with edge set the symmetric difference $M\triangle N$ and vertex set $V(M\triangle N)$. Then every vertex of $H$ belongs to at most two edges of $H$. Moreover, if a vertex is in two edges, then one edge is in $M\setminus N$ and the other edge is in $N\setminus M$. Therefore connected components of $H$ are paths and even cycles. Moreover, since $k$ is the matching number of $G$, it follows that no connected component of $H$ is an odd path. Indeed, if $H$ has an odd path with consecutive edges $e_1,\dots,e_{2p+1}$, we can assume that $e_1,e_3,\dots,e_{2p+1}\in M\setminus N$ and $e_2,e_4,\dots,e_{2p}\in N\setminus M$. But then $(N\setminus\{e_2,e_4,\dots,e_{2p}\})\cup\{e_1,e_3,\dots,e_{2p+1}\}$ is a $(k+1)$-matching.

    We verify the exchange property for $u$ and $v$ and this will imply statement (b). Let $z$ be a variable such that $\deg_zu>\deg_zv$. Then, Lemma~\ref{lem:degree of variable bigger than 1} implies that $z$ does not divide $v$. Then $z$ is the endpoint of an even path component of $H$. Let $y$ be the other endpoint of the path. Then $y\notin V(M)$. Since $y$ does not divide $u$, it follows from Lemma~\ref{lem:degree of variable bigger than 1} that $\deg_zu=\deg_yv=1$. Moreover, Lemma~\ref{lem:degree of variable bigger than 1} implies that $\deg_xu=\deg_xv$ for any vertex $x$ in the path. Without loss of generality, suppose that $e_1,f_1,\dots ,e_q,f_q$ are the consecutive edges of the path. Then exchange property is satisfied because $N_{f_1}\dots N_{f_q}u/M_{e_1}\dots M_{e_q}=yu/z \in G(I(\mathcal{D})^{[k]})$.
    \end{proof}
    
In \cite{EFnote} we classify all forests that satisfy the equivalent statements in the above theorem.

\begin{Example}
    \rm Let $\mathcal{D}$ be a weighted oriented graph whose underlying graph $G$ is an odd cycle, with vertex set $V(G)=[2k+1]$ and edge set $$E(G)=\{\{1,2\},\{2,3\},\dots,\{2k,2k+1\},\{2k+1,1\}\}.$$
    It is well--known that $\nu(G)=k$. If $I(\mathcal{D})=I(G)$, then by Theorem \ref{Thm:I^[nu(I)]Polymatroidal}, $I(\mathcal{D})^{[k]}$ is linearly related. We claim that $I(\mathcal{D})^{[k]}$ is not linearly related if $I(\mathcal{D})\neq I(G)$. Assume for a contradiction $I(\mathcal{D})^{[k]}$ is linearly related but $I(\mathcal{D})\neq I(G)$. Then without loss of generality, we may assume that $(2,1)\in E(\mathcal{D})$ and $w(1)>1$. Then, Lemma~\ref{lem:degree of variable bigger than 1} implies that all generators of $I(\mathcal{D})^{[k]}$ have $x_1$-degree bigger than 1. However, if we consider the $k$-matching $M=\{\{2,3\},\{4,5\},\dots,\{2k,2k+1\}\}$ of undirected edges of $G$, then there is a unique generator $v$ of $I(\mathcal{D})^{[k]}$ whose support is $V(M)$ and so $\deg_{x_1}(v)=0$, which is absurd.
\end{Example}

\begin{Example}
    \rm In Theorem \ref{Thm:I(D)linRel}, the condition that every subgraph of $G$ has at most one perfect matching is crucial. For example, let $\mathcal{D}$ be a weighted oriented graph with $I(\mathcal{D})=(x_1x_2^2, x_2x_3^2, x_2x_4^2, x_3x_1^2, x_3x_4^2, x_4x_1^2)$. Then $I(\mathcal{D})^{[2]}$ has a linear resolution but it is not polymatroidal. On the other hand, we do not know the answer to the following question:
\end{Example}

\begin{Question}
    Let $\mathcal{D}$ be a weighted oriented graph with $I(\mathcal{D})\neq I(G)$ where $G$ is the underlying graph. Suppose that $I(\mathcal{D})^{[k]}$ is linearly related. Then, does $I(\mathcal{D})^{[k]}$ have a linear resolution?
\end{Question}

If $\mathcal{D}$ is a connected weighted oriented graph with $I(\mathcal{D})\neq I(G)$, then the above question has a positive answer for $k=1$ by \cite[Theorem~3.5]{BDS23}.

\end{document}